\documentclass{amsart}

\usepackage[fleqn,tbtags]{mathtools}
\mathtoolsset{showonlyrefs,showmanualtags}

\usepackage{amssymb, amsmath}
\usepackage{mathrsfs}
\usepackage{amscd}
\usepackage[active]{srcltx}
\usepackage{verbatim}

\usepackage[colorlinks,linkcolor={blue},citecolor={blue},urlcolor={red},]{hyperref}

\usepackage{newsymbol}
\let\mathcal \undefined
\def\mathcal{\mathscr}

\let\emptyset \undefined
\let\ge       \undefined
\let\le       \undefined
\newsymbol\le          1336  \let\leq\le
\newsymbol\ge          133E  
\newsymbol\emptyset    203F
\newsymbol\notle       230A
\newsymbol\notge       230B

\theoremstyle{plain}
\newtheorem{theorem}{Theorem}[section]
\theoremstyle{remark}
\newtheorem{remark}[theorem]{Remark}
\newtheorem{example}[theorem]{Example}
\theoremstyle{plain}
\newtheorem{corollary}[theorem]{Corollary}
\newtheorem{lemma}[theorem]{Lemma}
\newtheorem{proposition}[theorem]{Proposition}

\numberwithin{equation}{section}

\def\N{{\mathbb N}}

\def\R{{\mathbb R}}

\def\dbP{{\mathbb{P}}}

\def\3n{\negthinspace \negthinspace \negthinspace }
\def\2n{\negthinspace \negthinspace }
\def\1n{\negthinspace }

\def\Om{\Omega}
\def\om{\omega}
\def\mE{{\mathbb{E}}}

\def\wt{\widetilde}

\def\({\Big (}
\def\){\Big )}
\def\[{\Big[}
\def\]{\Big]}

\def\={\buildrel \triangle \over =}

\newcommand{\E}{{\mathbb E}}
\renewcommand{\P}{{\mathbb P}}
\newcommand{\F}{{\mathcal F}}

\newcommand{\ud}{\,{\rm d}}

\newcommand{\beq}{\begin{equation}}
\newcommand{\eeq}{\end{equation}}
\newcommand{\bal}{\begin{aligned}}
\newcommand{\eal}{\end{aligned}}
\newcommand{\ben}{\begin{enumerate}}
\newcommand{\een}{\end{enumerate}}
\newcommand{\bit}{\begin{itemize}}
\newcommand{\eit}{\end{itemize}}

\newcommand{\bth}{\begin{theorem}}
\renewcommand{\eth}{\end{theorem}}
\newcommand{\bpr}{\begin{proposition}}
\newcommand{\epr}{\end{proposition}}
\newcommand{\ble}{\begin{lemma}}
\newcommand{\ele}{\end{lemma}}
\newcommand{\bpf}{\begin{proof}}
\newcommand{\epf}{\end{proof}}
\newcommand{\bex}{\begin{example}}
\newcommand{\eex}{\end{example}}
\newcommand{\bre}{\begin{example}}
\newcommand{\ere}{\end{example}}

\newcommand{\calA}{{\mathscr A}}
\newcommand{\calB}{{\mathscr B}}
\newcommand{\calC}{{\mathscr C}}
\newcommand{\calF}{{\mathscr F}}
\newcommand{\n}{\Vert}
\newcommand{\one}{{{\bf 1}}}

\newcommand{\s}{^*}
\newcommand{\lb}{\langle}
\newcommand{\rb}{\rangle}

\newcommand{\Ran}{{\mathsf{R}}}
\newcommand{\Ker}{{\mathsf{N}}}

\begin{document}

 \title{On conditional expectations in $L^p(\mu;L^q(\nu;X))$ }
 \author{Qi L\"u}
 \address{School of Mathematics, Sichuan University, Chengdu 610064, China}
 \email{lu@scu.edu.cn}
 \author{Jan van Neerven}
 \address{Delft University of Technology, P.O. Box 5031, 2600 GA Delft, The Netherlands}
 \email{J.M.A.M.vanNeerven@TUDelft.nl}
 \date{\today}
\thanks{Qi L\"u is supported by the NSF of China under
grants 11471231, the Fundamental Research Funds
for the Central Universities in China under
grant 2015SCU04A02, and Grant MTM2014-52347 of
the MICINN, Spain.}
\begin{abstract}
Let  $(A,\calA,\mu)$ and $(B,\calB,\nu)$ be
probability spaces, let $\F$ be a sub-$\sigma$-algebra
of the product $\sigma$-algebra $\calA\times \calB$,
let $X$ be a Banach space and let $1< p,q< \infty$.
We obtain necessary and sufficient conditions in order that the
conditional expectation with respect to $\F$  defines a
bounded linear operator from $L^p(\mu;L^q(\nu;X))$ onto
$L^p_\calF(\mu;L^q(\nu;X))$, the closed subspace
in $L^p(\mu;L^q(\nu;X))$ of all functions having
a strongly $\calF$-measurable representative.
\end{abstract}

\maketitle

\section{Introduction}
Let $(A,\calA,\mu)$ and $(B,\calB,\nu)$ be probability spaces, 
$\calF$ a sub-$\sigma$-algebra of the product $\sigma$-algebra 
$\calA\times \calB$ in $A\times B$, and $X$ a Banach space.
For $1\le p,q\le \infty$ we define
$L_\calF^p(\mu; L^q(\nu;X))$ to be the closed
subspace in $L^p(\mu; L^q(\nu;X))$ consisting of
those  functions which have a strongly
$\calF$-measurable representative. It is easy to
see (e.g., by using \cite[Corollary 1.7]{HNVW}) that
$$ L_\calF^p(\mu; L^q(\nu;X))= L^p(\mu; L^q(\nu;X))
\cap L_\calF^1( \mu\times \nu;X).$$
Furthermore, $L_\calF^p(\mu; L^q(\nu;X))$ is
closed in $L^p(\mu; L^q(\nu;X))$. Indeed, if
$f_n\to f$ in $L^p(\mu; L^q(\nu;X))$ with each
$f_n$ in $L_\calF^p(\mu; L^q(\nu;X))$, then also
$f_n\to f$ in $L^1(\mu\times \nu;X)$, and
therefore $f\in L_\calF^1(\mu\times\nu;X)$. The
reader is referred to \cite{DU, HNVW} for the basic
theory of the Lebesgue-Bochner spaces and
conditional expectations in these spaces. The
same reference contains some standard results
concerning the Radon-Nikod\'ym property that
will be needed later on.

The aim of this paper is to provide a necessary and sufficient 
condition in order that 
conditional expectation $\E(\cdot|\calF)$
restrict to a bounded linear operator on
$L^p(\mu;L^q(\nu;X))$ when 
$1<p,q<\infty$. We also show that
$\E(\cdot|\calF)$ need not to be contractive.
An
example is given which shows that this result
does not extend to the pair $p=\infty$, $q=2$.

Characterisations of conditional expectation operators on general classes of Banach function spaces
$E$ (and their vector-valued counterparts) have been given by various authors (see, e.g., \cite{DHP} and the references therein), but these works
usually {\em assume} that a bounded operator $T: E\to E$ is given and investigate under
what circumstances it is a conditional expectation operator. We have not been able to find any paper addressing the
problem of establishing sufficient conditions for conditional expectation operators to act in
concrete Banach function spaces such as the mixed-norm $L^p(L^q)$-spaces investigated here.

\section{Results}

Throughout this section, $(A,\calA,\mu)$ and $(B,\calB,\nu)$ are probability spaces.
If $1\le p,q\le \infty$, their conjugates $1\le p',q'\le \infty$ are defined by $\frac1p+\frac1{p'} =1$ and $ \frac1q+\frac1{q'} =1$.

It is clear that every $f\in L_{\calF}^p(\mu; L^q(\nu))$ induces a functional
$\phi_f\in (L_{\calF}^{p'}(\mu;L^{q'}\!(\nu)))\s$ in a canonical way, and the resulting mapping $f\mapsto \phi_f$ is contractive.
The first main result of this note reads as follows.

\begin{theorem}\label{thm1} Let $1 <p\le\infty$ and $1<q\le \infty$.
If $f\mapsto \phi_f$ establishes an isomorphism of Banach spaces 
$$L_\calF^p(\mu; L^q(\nu)) \simeq (L_{\calF}^{p'}(\mu;L^{q'}\!(\nu)))\s ,$$ 
then for any Banach space $X$
the conditional expectation operator $\E(\cdot|\calF)$ on $L^1(\mu\times\nu;X)$ 
restricts to a bounded projection on $L^p(\mu; L^q(\nu;X))$.
\end{theorem}

\begin{proof} We will show that  $\E(f|\calF)\in L^p(\mu; L^q(\nu;X))$
for all $f\in L^p(\mu; L^q(\nu;X))$. A standard closed graph argument then gives the boundedness
of $\E(\cdot|\calF)$ as an operator in $L^p(\mu; L^q(\nu;X))$.

Since $\n \E(f|\calF)\n_X \le \E (\n f\n_X | \calF)$ $\mu\times\nu$-almost everywhere, it suffices
to prove that $\E(g|\calF)\in L^p(\mu; L^q(\nu))$
for all $g\in L^p(\mu; L^q(\nu))$.
To prove the latter, consider the inclusion mapping
$$I: L_\calF^{p'}(\mu; L^{q'}\!(\nu)) \to L^{p'}\!(\mu;L^{q'}\!(\nu)).$$
Every $g\in L^p(\mu; L^q(\nu))$
defines an element of $(L^{p'}\!(\mu; L^{q'}\!(\nu)))\s$ in the natural way and we have, for all $F\in \calF$,
$$ \lb \one_F, I\s g\rb = \lb I \one_F,g\rb = \int_{F} g \ud\mu\times \nu  .$$
The implicit
use of Fubini's theorem to rewrite the double integral over $A$ and $B$ as an integral over $A\times B$
in the second equality is justified by non-negativity, writing
$g = g^+ - g^-$ and considering these functions separately.
On the other hand, viewing $g$ and $\one_F$ as elements of $L^1(\mu\times\nu)$
and $L^\infty(\mu\times\nu)$ respectively, we have
$$  \int_{F} g \ud\mu\times \nu
= \int_{F} \E(g|\calF) \ud\mu\times \nu
= \lb \one_F, \E(g|\calF) \rb.$$
We conclude that $\lb \one_F, I\s g\rb = \lb \E(g|\calF),\one_F \rb$, where on the left the duality is
between $L^{p'}\!(\mu; L^{q'}\!(\nu))$ and its dual, and on the right between $L^1(\mu\times \nu)$ and $L^\infty(\mu\times\nu)$.
Passing to linear combinations of indicators, it follows that
$$
\sup_\phi|\lb \phi, I\s g\rb|
 = \sup_\phi|\lb \E(g|\calF), \phi \rb| = \n  \E(g|\calF)\n_1 < \infty,
$$
where both suprema run over the simple functions $\phi$ in $L^\infty_\calF(\mu\times\nu)$ of norm $\le 1$. Denoting their closure by $L^\infty_{0,\calF}(\mu\times\nu)$,
it follows that $I\s g$ defines an element of $(L^\infty_{0,\calF}(\mu\times\nu))\s$.
This identification is one-to-one: for if $\lb \phi, I\s g\rb = 0$ for all simple $\calF$-measurable functions $\phi$, then $\lb \phi, I\s g\rb =0$ for all $\phi\in L_{\F}^{{ p'}}\!(\mu; L^{{ q'}}\!(\nu))$, noting that
the simple $\calF$-measurable functions are dense in $L_{\F}^{{ p'}}\!(\mu; L^{{ q'}}\!(\nu))$ (here we use that
${ p'}$ and ${ q'}$ are finite).

As an element of $(L^\infty_{0,\calF}(\mu\times\nu))\s$, $I\s g$ equals the function $\E(g|\calF)$, viewed as an element in the same space. Since the embedding
of $ L_\calF^1(\mu\times\nu)$ into $(L^\infty_{0,\calF}(\mu\times\nu))\s$ is isometric, it follows that $I\s g = \E(g|\calF) \in L_\calF^1(\mu\times\nu)$.
Since $I\s g \in (L_{\calF}^{p'}(\mu;L^{q'}\!(\nu)))\s $, by the assumption of the theorem we may identify
$I\s g$ with a function in $L^p(\mu;L^q(\nu))$.
We conclude that
$\E(g|\calF) = I\s g\in L_\calF^{p}(\mu; L^{q}(\nu))$.
\end{proof}

{
If we make a stronger assumption, more can be said:

\begin{theorem}\label{thm2}
Suppose that $1<p,q<\infty$ and let $X$ be a non-zero Banach space.
Then the following assertions are equivalent:
\begin{enumerate}
 \item[\rm(1)] the conditional expectation operator $\E(\cdot|\calF)$ restricts to a bounded projection on the space $L^p(\mu; L^q(\nu;X))$;
 \item[\rm(2)] { the conditional expectation operator $\E(\cdot|\calF)$ restricts to a bounded projection on the space $L^{p'}\!(\mu; L^{q'}\!(\nu;X))$;}
 \item[\rm(3)] $f\mapsto \phi_f$ induces an isomorphism of Banach spaces 
$$L_\calF^p(\mu; L^q(\nu)) \simeq (L_{\calF}^{p'}(\mu;L^{q'}\!(\nu)))\s .$$
\end{enumerate}
\end{theorem}

\begin{remark}
In \cite{LYZ} it is shown that condition (3) is satisfied if 
\begin{equation}\label{LYZ}
\hbox{  $I \times \E_\nu$ maps $L^1_\calF(\mu\times\nu)$ into itself.} 
\end{equation}
Here
$\E_\nu$ denotes the bounded operator on $L^1(\nu)$ defined by
$$ \E_\nu f := (\E_\nu f)\one,$$
with $\E_\nu f = \int f\ud \nu$.
\end{remark}

The proof of Theorem \ref{thm2} is based on the following elementary lemma.

\begin{lemma}\label{lem:proj}
 Let $P$ be a bounded projection on a Banach space $X$. Let $X_0 = \Ran(P)$, $X_1 = \Ker(P)$, $Y_0 = \Ran(P\s)$ and $Y_1 = \Ker(P\s)$, so that we have direct sum decompositions $X = X_0\oplus X_1$ and $X\s = Y_0\oplus Y_1$. Then 
we have natural isomorphisms of Banach spaces $X_0\s = Y_0$ and $X_1\s = Y_1$.
\end{lemma}

\begin{proof}[Proof of Theorem \ref{thm2}]
We have already proved (3)$\Rightarrow$(1). For proving (1)$\Rightarrow$(2)$\Rightarrow$(3)
there is no loss of generality in assuming that $X$ is the scalar field,
for instance by observing that the proof of Theorem \cite[Theorem 2.1.3]{HNVW} also 
works for mixed $L^p(L^q)$-spaces.

\smallskip
(1)$\Rightarrow$(2):\
The assumption (1) implies that
$L_\calF^{p}(\mu;L^{q}(\nu))$ is the range of the bounded projection $(\E(\cdot|\calF))$
in $L^{p}(\mu;L^{q}(\nu))$. Moreover, $\lb \E(f|\calF),g\rb = \lb f,\E(g|\calF)\rb$ for all $f\in L^p(\mu;L^q(\nu))$
and $g\in L^{p'}\!(\mu;L^{q'}\!(\nu))$, since this is true for $f$ and $g$ in the (dense) intersections of these spaces with $L^2(\mu\times \nu)$.
It follows that the conditional expectation $\E(\cdot|\calF)$ is bounded on $L^{p'}\!(\mu;L^{q'}\!(\nu)) =(L^{p}(\mu;L^{q}(\nu)) )\s $ and equals $(\E(\cdot|\calF))\s$. 
Clearly it is a projection and its range equals
$L_\calF^{p'}\!(\mu;L^{q'}\!(\nu))$.

\smallskip (2)$\Rightarrow$(3):\
This implication follows Lemma \ref{lem:proj}.
\end{proof}

Inspection of the proof of Theorem \ref{thm1} 
shows that if for all $f\in L_\calF^{p}(\mu; L^{q}(\nu))$ we have
$\n f\n_{L_\calF^p(\mu;L^q(\nu))} = \n f\n_{(L_\calF^{p'}(\mu; L^{q'}\!(\nu)))\s}$,
then $\E(\cdot|\calF)$ is contractive on $L^p(\mu;L^q(\nu))$.
The next example, due to Qiu \cite{Qiu}, shows that the conditional expectation, when it is bounded, may fail to be contractive.

\begin{example}\label{ex:Qiu}
Let $A = B = \{0,1\}$ with $\calA = \calB = \{\emptyset, \{0\}, \{1\},\{0,1\}\}$ and
$\mu = \nu$ the measure on $\{0,1\}$ that gives each point mass $\frac12$,
 and let $\F$ be the $\sigma$-algebra generated by the three sets $\{(0,1)\}$, $\{(1,1)\}$, $\{(0,0),(1,0)\}$.
If we think of $B$ as describing discrete `time', then $\F$ is the progressive $\sigma$-algebra
corresponding to the filtration $(\F_t)_{t\in\{0,1\}}$ in $A$ given by $\F_0 = \{\emptyset,\{0,1\}\}$
and $\F_1 = \{\emptyset, \{0\}, \{1\},\{0,1\}\}$.

Let $f: A\times B \to \R$ be defined by
$$
f(0,0) = 0, \quad f(1,0) = 1, \quad f(0,1) = 1, \quad f(1,1)=0.
$$
Then
$$
\E(f|\F)(0,0) = \frac12, \quad \E(f|\F)(1,0) = \frac12, \quad \E(f|\F)(0,1) = 1, \quad \E(f|\F)(1,1)=0.
$$
Hence in this example we have
\begin{equation*}
\begin{aligned}
\n f\n_{L^p(\mu;L^2(\nu))} & = \Big[\Big(\frac12\Big)^{p/2} + \Big(\frac12\Big)^{p/2}\Big]^{1/p},  \\
\n \E(f|\calF)\n_{L^p(\mu;L^2(\nu))} & =
\Big[\Big(\frac18\Big)^{p/2} +
\Big(\frac58\Big)^{p/2}\Big]^{1/p}.
\end{aligned}
\end{equation*}
Consequently, for large enough $p$
the conditional expectation fails to be contractive in $L^p(\mu;L^2(\nu))$.
\end{example}

We continue with two examples showing that 
the condition expectation operator on $L^1(\mu\times\nu)$ may fail to restrict to a bounded operator 
on $L^p(\mu;L^q(\nu))$. The first was communicated to us by Gilles Pisier.

\begin{example}\label{ex:Pisier}
Let $(A,\calA, \mu)$ and $(B,\calB, \nu)$ be probability spaces and let 
$(C,\calC,\P) = (A,\calA, \mu)\times (B,\calB, \nu)$ be their product.
Consider the infinite product $(C,\calC,\P)^\N = (C^\N, \calC^\N,\P^\N)$; with an obvious identification it may be identified with 
$ (A^\N,\calA^\N, \mu^\N)\times (B^\N,\calB^\N, \nu^\N)$. 

Consider the sub-$\sigma$-algebra $\calF^\N$ of $\calA^\N\times\calB^\N = \calC^\N$, where $\calF\subseteq \calA\times\calB = \calC$ is a given sub-$\sigma$-algebra.
Let $T:= \E(\cdot|\calF)$ and $T^\N := \E(\cdot|\calF^\N)$ be the conditional expectation operators on $L^1(\mu\times \nu)$ and $L^1(\mu^\N\times\nu^\N)$, respectively.
For a function $f\in L^\infty(\mu^\N\times \nu^\N)$ of the form $f = f_1\otimes\cdots\otimes f_N\otimes \one\otimes\one\otimes \hdots$
with $f_n\in L^1(\mu\times \nu)$ for all $n=1,\dots,N$, we have 
$$ T^\N f = Tf_1 \otimes\cdots\otimes T f_N\otimes \one\otimes\one\otimes \hdots$$
By an elementary computation,
$$ \n f\n_{L^p(\mu^\N;L^q(\nu^\N))} = \prod_{n=1}^N \n f_n \n_{L^p(\mu;L^q(\nu))}$$
and 
$$\n T^\N f\n_{L^p(\mu^\N;L^q(\nu^\N))} = \prod_{n=1}^N \n Tf_n \n_{L^p(\mu;L^q(\nu))}.$$
This being true for very $N\ge 1$ we see that $T^\N$ is bounded if and only if $T$ is contractive. Example \ref{ex:Qiu}, however, shows that 
the latter need not always be the case.
\end{example}

The second example is due to Tuomas Hyt\"onen:

\begin{example}\label{ex:Tuomas}
Let $\calB$ the Borel $\sigma$-algebra of $[0, 1)$. For $A \in \calB \times\calB$, let
$$\wt A := \{(y, x) : (x, y) \in A\}$$
and let
$$\calF := \{A \in \calB \times\calB :\, \wt A= A\}$$
be the symmetric sub-$\sigma$-algebra of the product $\sigma$-algebra. Then $\E(\cdot|\calF )$ does not
restrict to a bounded operator on $L^p(L^q):= L^p(0, 1; L^q (0, 1))$ when $p  \not = q$.
To see this let $\wt f(x, y) := f (y, x)$. One checks that
$$ \E(f |\calF ) = \frac12(f + \wt f) \ge \frac12 \wt f$$
if $f \ge 0$. In particular, $\E(\phi  \otimes \psi |\calF ) \ge  \frac12  \psi  \otimes \phi$  if $\phi , \psi  \ge 0$.
Let then $\phi  \in L^p(0, 1)$, $\psi  \in L^q (0, 1)$ be positive functions such that only one of
them is in $L^{p\vee q} (0, 1)$. If $f = \phi  \otimes \psi$, then
$$\n f \n_{L^p (L^q )} = \n \phi \n_{L^p} \n \psi \n_{L^q} < \infty$$
but
$$\n \E(f |\calF )\n_{ L^p (L^q )}\ge \frac12 
 \n \psi  \otimes \phi \n_{ L^p(L^q )} = \frac12\n \psi \n_{ L^p }\n \phi \n_{ L^q} .$$
If $p> q$, then $\n \psi \n_{ L^p} = \infty$, and if $p < q$, then $\n \phi \n _{L^q} = \infty$, so that in either case
$\n \E(f |\calF )\n_{ L^p (L^q ) }= \infty$. 
\end{example}

Let us check that \eqref{LYZ} fails in the above examples. As in Example \ref{ex:Qiu}
let $A = B = \{0,1\}$ with $\calA = \calB = \{\emptyset, \{0\}, \{1\},\{0,1\}\}$,
$\mu = \nu$ the measure on $\{0,1\}$ that gives each point mass $\frac12$, and  $\F$ 
the $\sigma$-algebra generated by the three sets $\{(0,1)\}$, $\{(1,1)\}$, $\{(0,0),(1,0)\}$.
Let  $f: A\times B \to \R$ be defined by
$$
f(0,0) = 1, \quad f(1,0) = 1, \quad f(0,1) = 0, \quad f(1,1)=1.
$$
This function is $\calF$-measurable, but $(I\otimes \E_\nu)f$ is not:
$$(I\otimes \E_\nu)f(0,0) = \frac12, \ \ (I\otimes \E_\nu)f(1,0) = 1, \ \
(I\otimes \E_\nu)f(0,1) = \frac12, \ \ (I\otimes \E_\nu)f(1,1) = 1. 
$$
Thus \eqref{LYZ} fails in Example \ref{ex:Qiu}.
It is clear that if we start from this example,  \eqref{LYZ} also fails in Example \ref{ex:Pisier}. 
In Example \ref{ex:Tuomas} \eqref{LYZ} also fails, for obvious reasons.

An interesting example where condition 
\eqref{LYZ} is satisfied is the case when $A =[0,1]$ is the unit interval, $B = \Omega$ a probability space,
and $\calF = {\mathscr P}$ the progressive $\sigma$-algebra in $[0,1]\times \Omega$. 
From Theorem \ref{thm1} we therefore obtain the following result:

\begin{corollary}\label{cor:LYZ} For all $1<p,q<\infty$ and all Banach spaces $X$, the conditional expectation
with respect to the progressive $\sigma$-algebra on $[0,1]\times\Omega$ is bounded on $L^p(0,1;L^q(\Om;X))$. 
\end{corollary}

This quoted result of \cite{LYZ} plays
an important role in the study of well-posedness
and control problems for stochastic partial
differential equations. For example, in
\cite{Lu2}, it is used to show the well-posedness of
stochastic Schr\"odinger equations with
non-homogeneous boundary conditions in the sense
of transposition solutions,
in \cite{Lu1} it is applied
to obtain a relationship between
null controllability of stochastic heat
equations, and in \cite{LYZ} it is used to
establish a Pontryagin type maximum for
controlled stochastic evolution equations with
non-convex control domain.

As a consequence of (a special case of)
\cite[Theorem A.3]{DY}  we obtain that 
the assumptions of Theorem \ref{thm1} are also satisfied for progressive $\sigma$-algebra
$\calF = {\mathscr P}$ if we replace $L^p(0,1;L^q(\Om;X))$ by $L^p(\Om;L^q(0,1;X))$.
The quoted theorem is stated in terms of the predictable $\sigma$-algebra $\mathscr{G}$. However, since every progressively measurable set
$P\in\mathscr{P}$ is of the form $P = G \Delta N$ with $G\in \mathscr{G}$ and $N$ a null set in the product $\sigma$-algebra $\mathscr{F}\times \mathscr{B}([0,1])$ (see \cite[Lemma 3.5]{CW}),
we have $L_\mathscr{G}^p(\Omega;L^q(0,1;X)) = L_\mathscr{P}^p(\Omega;L^q(0,1;X))$. Therefore, \cite[Theorem A.3]{DY} remains true if 
we replace the predictable $\sigma$-algebra by the progressive $\sigma$-algebra and we obtain the following result:

\begin{corollary} For all $1<p,q<\infty$ and all Banach spaces $X$,
the conditional expectation
with respect to the progressive $\sigma$-algebra on $\Omega\times[0,1]$ is bounded on $L^p(\Om;L^q(0,1;X))$.
\end{corollary}
\begin{proof}
In the scalar-valued case we apply \cite[Theorem A.3]{DY} (with $J$ a singleton).
The vector-valued case then follows from the observation, already made in the proof of Theorem \ref{thm2}, that Theorem \cite[Theorem 2.1.3]{HNVW} also holds for mixed $L^p(L^q)$-spaces.
\end{proof}


Our final example shows that condition (2) in Theorem \ref{thm2} fails for
the pair $p=1$, $q=2$ 
even when $X$ is the scalar field. 

\begin{example}
Let $\{\mathscr{F}_t\}_{t\in[0,1]}$ be the filtration generated by a
one-dimensional standard Brownian motion
$\{W(t)\}_{t\in[0,1]}$ defined on a probability space $(\Omega,\mathscr{F},\dbP)$. Let $\mathscr{P}$ be the associated progressive $\sigma$-algebra
on $\Om\times [0,1]$. We will show that
\begin{equation*}
L^\infty_\mathscr{P}(\Om;L^2(0,1)) \subsetneq (L^1_\mathscr{P}(\Om;L^2(0,1)))^*
\end{equation*}
in the sense that the former is contained isometrically as a
{\em proper}
closed subspace of the latter.

For $v\in L^1_\mathscr{P}(\Om;L^2(0,1))$ consider the
solution $x$ to the following problem:
\begin{equation}\label{eq2}
\left\{\begin{aligned}
\ud x(t) & =v(t)\ud W(t), \quad t\in [0,1],\\
    x(0) & =0.
\end{aligned}
\right.
\end{equation}
By the classical well-posedness theory of SDEs
(e.g. \cite[Chapter V, Section 3]{Protter}),
$x\in L^1_\mathscr{P}(\Om;C([0,1]))$ and
\begin{equation}\label{eq3}
\n x\n_{L^1_\mathscr{P}(\Om;C([0,1]))}\leq
C\n v\n _{L^1_\mathscr{P}(\Om;L^2(0,1))}
\end{equation}
for some constant $C$ independent of $v$. Let $\xi\in L^{\infty}_{\mathscr{F}_1}(\Om)$.
Define a linear functional $L$ on
$L^1_\mathscr{P}(\Om;L^2(0,1))$ as follows:
$$
L(v):=\mE(\xi x(1)).
$$
By \eqref{eq3}, $L$ is bounded.
Suppose now, for a contradiction, that $(L^1_\mathscr{P}(\Om;L^2(0,1)))^* =
L^{\infty}_\mathscr{P}(\Om;L^2(0,1))$
with equivalent norms. Then there is an $f\in
L^{\infty}_\mathscr{P}(\Om;L^2(0,1))$ such that
\begin{equation}\label{eq4}
L(v)=\mE\int_0^1 f(t)v(t)\ud t
\end{equation}
for all $v\in L^1_\mathscr{P}(\Om;L^2(0,1))$.
On the other hand, by the martingale representation
theorem there is a $g\in L_\mathscr{P}^2(\Om;L^2(0,1))$
such that
\begin{equation}\label{eq:MRT}
\xi = \mE(\xi) + \int_0^1 g(t)\ud W(t).
\end{equation}
Take now $v\in L_\mathscr{P}^2(\Om;L^2(0,1))$ in \eqref{eq2}. Then by
It\^o's formula,
\begin{equation}\label{eq5}
\mE(\xi x(1))= \mE\int_0^1 g(t)v(t)\ud t.
\end{equation}
Since \eqref{eq4} and \eqref{eq5} hold for all
$v\in L_\mathscr{P}^2(\Om;L^2(0,1))$,
it follows that $f=g$ for almost all $(t,\om)\in
(0,1)\times\Om$.  Hence, $g\in
L^{\infty}_\mathscr{P}(\Om;L^2(0,1))$.  This
leads to a contradiction, since it would imply
that the isometry from $\{\xi\in L_{\calF_1}^2(\Om): \, \E \xi = 0\}$ into $L_{\mathscr{P}}^2(\Om;L^2(0,1))$
given by \eqref{eq:MRT} sends
$\{\xi\in L_{\calF_1}^\infty(\Om):\, \E \xi = 0\}$ into
$L^{\infty}_\mathscr{P}(\Om;L^2(0,1))$. This is
known to be false (see, e.g., \cite[Lemma A.1]{Fd}).
\end{example}

It would be interesting to determine an explicit representation for the dual of $L_\mathscr{P}^1(\Om;L^2(0,1))$.

\begin{remark}
In \cite{LYZ}, the authors proved that
$(L^1_\mathscr{P}(0,1;L^2(\Om)))^* =
L^\infty_\mathscr{P}(0,1;L^2(\Om)).$
 It seems that this result cannot be
obtained by the method in this paper.
\end{remark}

{\em Acknowledgment} -- The authors thank Gilles Pisier for pointing out an error in an earlier version of the paper and 
communicating to us Example \ref{ex:Pisier} and Tuomas Hyt\"onen for showing us Example \ref{ex:Tuomas}.

\end{document}